\documentclass[11pt]{amsart}

\usepackage{amsmath, amsfonts, amssymb,amsthm}
\usepackage{amstext}
\usepackage{mathrsfs}

\usepackage{float,epsf,subfigure}
\usepackage[all,cmtip]{xy}
\usepackage{hyperref}
\usepackage{algorithm}
\usepackage{algorithmic}
\usepackage{mathtools}
\usepackage{float}
\usepackage{bm}
\theoremstyle{plain}
\newtheorem{theorem}{Theorem}[section]

\newtheorem{def-thm}[theorem]{Definition-Theorem}
\newtheorem{lemma}[theorem]{Lemma}

\newtheorem*{thai}{Theorem I}
\newtheorem*{thaii}{Theorem II}
\newtheorem*{thaiii}{Theorem III}
\newtheorem*{thaiv}{Theorem IV}
\newtheorem*{thav}{Theorem V}
\newtheorem*{thavi}{Theorem VI}
\newtheorem*{thavii}{Theorem VII}

\theoremstyle{definition}

\def\min{\mathop{\mathrm{min}}}

{

\begin{document}
\title[Fermat functional equations over Riemann surfaces]{Fermat functional equations over Riemann surfaces}
\author[X.J. Dong and L.W. Liao and K. Liu] {Xianjing Dong and Liangwen Liao and Kai Liu}

\address{Academy of Mathematics and Systems Sciences \\ Chinese Academy of Sciences \\ Beijing, 100190, P.R. China}
\email{xjdong@amss.ac.cn}

\address{Department of Mathematics \\ Nanjing University \\ Nanjing, 210093, P.R. China}
\email{maliao@nju.edu.cn}

\address{Department of Mathematics \\ Nanchang University \\ Nanchang, 330031, P.R. China}
\email{liukai@ncu.edu.cn}


\subjclass[2010]{30D05, 30D35.} \keywords{Fermat functional equations;  Meromorphic solutions; Riemann surfaces; Jet differentials; Vanishing Theorem; Logarithmic Derivative Lemma}
\date{}
\maketitle \thispagestyle{empty} \setcounter{page}{1}

\begin{abstract} We investigate the existence of non-trivial holomorphic and meromorphic solutions of Fermat functional equations over an  open Riemann surface $S$.
When $S$ is hyperbolic, we prove that any $k$-term Fermat functional equation always exists non-trivial holomorphic and meromorphic solution. 
When $S$ is a general  open Riemann surface,  
we  prove that every non-trivial holomorphic or meromorphic solution  satisfies a growth condition, provided that  the power exponents of  the equations are bigger  than some certain positive integers.
\end{abstract}

\vskip\baselineskip

\setlength\arraycolsep{2pt}

\section{Introduction}

It is   natural to study holomorphic and meromorphic solutions of  a $k$-term Fermat functional equation
\begin{equation}\label{asdf}
f_1^n+\cdots+f_k^n=1
\end{equation}
over an open Riemann surface. This is a  generalization of  Fermat functional equations over $\mathbb C$ which were treated  by  many researchers such as Baker \cite{ba}, Gross \cite{gros, gro},  Gundersen-Tohge \cite{gund, gun, gun-toh},  Hayman \cite{hay},  Iyer \cite{iye},  Ishizaki \cite{ish}, Lehmer \cite{leh},  Toda \cite{tod} and Yang \cite{yang}, etc.
We first  review some important developments in the $\mathbb C$-case. When $k=2,$ Iyer \cite{iye} proved the non-existence of non-trivial holomorphic solutions for $n\geq3$ and proved that all holomorphic solutions are of the form $f_1=\cos\circ\alpha, f_2=\sin\circ\alpha$ for $n=2,$ here $\alpha$ is a holomorphic function on $\mathbb C.$ Gross \cite{gros} proved the non-existence of 
non-trivial meromorphic solutions for $n\geq4$ and proved that all meromorphic solutions are of the form $f_1=2\beta/(1+\beta^2), f_2=(1-\beta^2)/(1+\beta^2)$ for $n=2,$ here $\beta$ is a meromorphic function on $\mathbb C.$ Baker \cite{ba} proved that every meromorphic solution is of the form $f_1=p\circ\alpha, f_2=q\circ\alpha,$ here $\alpha$ is a holomorphic function on $\mathbb C$ and  
$$p=\frac{1}{2\wp}\Big(1-3^{-1/2}\wp'\Big), \ \  \
 q=\frac{\varpi}{2\wp}\Big(1+3^{-1/2}\wp'\Big),$$
 where $\varpi$ is a cube root of unity, $\wp$ is a Weierstrass elliptic function satisfying  $(\wp')^2=4\wp^3-1.$ When $k=3,$ Hayman \cite{hay} proved the non-existence of non-trivial holomorphic solutions for $n\geq7$  and  meromorphic solutions for  $n\geq9.$ Some examples of non-trivial holomorphic solutions for $2\leq n\leq 5$ as well as  
  meromorphic solutions for $2\leq n\leq 6$ were constructed by Gundersen-Tohge \cite{gund, gun, gun-toh},  Green \cite{green} and Lehmer \cite{leh}, etc. However, up to now, we still don't know  that whether there exist  non-trivial holomorphic solutions when $n=6$ and non-holomorphic meromorphic solutions when $n=7,8.$

In this paper, we treat Fermat functional equations over Riemann surfaces from a geometric point of view. More specifically,  one describes the  existence of non-trivial holomorphic and meromorphic solutions via Gauss curvature.

Let $S$ be an open (connected) Riemann surface. Due to the uniformization theorem,  the (analytic) universal covering $\tilde S$ of $S$ is either $\mathbb C$ or $\mathbb D.$
Thus, one can  endow $S$ with a complete  Hermitian metric $ds^2=2gdzd\bar{z}$   such that  
 the 
   Gauss curvature
$K_S\leq0$  associated to $g,$  here $K_S$ is given  by
$$K_S=-\frac{1}{4}\Delta_S\log g=-\frac{1}{g}\frac{\partial^2\log g}{\partial z\partial\bar z}.$$
\ \ \ \  Fix $o\in S$ as a reference point. Denote by $D_o(r)$  the geodesic disc centered at $o$ with radius $r,$ and by $\partial D_o(r)$  the boundary of $D_o(r).$
By Sard's theorem, $\partial D_o(r)$ is a submanifold of $S$ for almost all $r>0.$
Set
\begin{equation}\label{kappa}
  \kappa(t)=\min\big\{K_S(x): x\in \overline{D_o(t)}\big\},
\end{equation}
which gives a non-positive and decreasing and continuous function  on $[0,\infty).$  
Let $g_r(o,x)$  be the Green function of $\Delta_S/2$ for  $D_o(r)$ with Dirichlet boundary condition and a pole at $o$, namely
$$-\frac{1}{2}\Delta_{S}g_r(o,x)=\delta_o(x), \ x\in D_o(r); \ \ g_r(o,x)=0, \ x\in \partial D_o(r).$$
For an integer $\nu$ with $1\leq\nu\leq k,$ define
$$\mathfrak T_{f_1,\cdots, f_\nu}(r):=\frac{1}{4}\int_{D_o(r)}g_r(o,x)\Delta_S\log(1+ |f_1(x)|^2+\cdots+|f_\nu(x)|^2)dV(x),$$
which describes  the growth of $f_1,\cdots,f_\nu.$

Firstly, we investigate the  equation (\ref{asdf}) over $S$. 
\begin{thai}\label{sss}  If $S$ is hyperbolic, then there  exist  non-trivial  holomorphic and meromorphic  solutions of  $(\ref{asdf})$ 
\end{thai}

\begin{thaii} When  $k=2,$  there exist no non-trivial holomorphic solutions for $n\geq 3$ and no non-trivial meromorphic solutions for
 $n\geq 4$ of $(\ref{asdf})$ over $S$ satisfying the growth condition
$$\liminf_{r\rightarrow\infty}\frac{\kappa(r)r^2}{ \mathfrak T_{f_1,f_2}(r)}=0.$$
 In particular,  the   conclusions hold for $\tilde S=\mathbb C$ without  growth condition.
\end{thaii}

\begin{thaiii} When  $k=3,$  there exist no non-trivial holomorphic solutions for $n\geq 7$ 
 and no non-trivial meromorphic solutions for $n\geq 9$ 
of $(\ref{asdf})$ over $S$ satisfying the growth condition
$$\liminf_{r\rightarrow\infty}\frac{\kappa(r)r^2}{\mathfrak T_{f_1,f_2, f_3}(r)}=0.$$
 In particular,  the   conclusions hold for $\tilde S=\mathbb C$ without  growth condition.
\end{thaiii}
Secondly,  we  consider the generalized $k$-term Fermat functional equation 
\begin{equation}\label{asdf1}
f_1^{n_1}+\cdots+f_k^{n_k}=1
\end{equation}
over $S.$

\begin{thaiv}\label{sss1}  If $S$ is hyperbolic, then there  exist  non-trivial  holomorphic and meromorphic  solutions of  $(\ref{asdf1}).$ 
\end{thaiv}
\begin{thav} When  $k=2,$  there exist no non-trivial holomorphic solutions for $1/n_1+1/n_2<1$ 
of $(\ref{asdf1})$ over $S$ satisfying the growth condition
$$\liminf_{r\rightarrow\infty}\frac{\kappa(r)r^2}{\min\{\mathfrak T_{f_1}(r), \mathfrak T_{f_2}(r)\}}=0;$$
there exist  no non-trivial meromorphic solutions for 
 $1/n_1+1/n_2 \leq 1/2$ of $(\ref{asdf1})$ over $S$ satisfying the growth condition
$$\liminf_{r\rightarrow\infty}\frac{\kappa(r)r^2}{\mathfrak T_{f_1,f_2}(r)}=0.$$
 In particular,  the  conclusions hold for $\tilde S=\mathbb C$ without  growth condition.
\end{thav}

\begin{thavi} When  $k=3,$  there exist no non-trivial holomorphic solutions for  $1/n_1+1/n_2+1/n_3<1/2$  of $(\ref{asdf1})$ over $S$ satisfying the growth condition
$$\liminf_{r\rightarrow\infty}\frac{\kappa(r)r^2}{\min\{\mathfrak T_{f_1}(r), \mathfrak T_{f_2}(r), \mathfrak T_{f_3}(r)\}}=0;$$
there exist no non-trivial meromorphic solutions  for  $1/n_1+1/n_2+1/n_3 \leq 1/3$  of $(\ref{asdf1})$ over $S$ satisfying the growth condition
$$\liminf_{r\rightarrow\infty}\frac{\kappa(r)r^2}{\mathfrak T_{f_1,f_2, f_3}(r)}=0.$$
 In particular,  the  conclusions hold for $\tilde S=\mathbb C$ without  growth condition.
\end{thavi}

Finally,  we treat the Fermat  functional equation for small functions
\begin{equation}\label{ggg}
\alpha_1f_1^{n_1}+\cdots+\alpha_kf^{n_k}_k=1, \ \ k\geq2
\end{equation}
over $S,$ where $\alpha_j$ is a small function with respect to  $f_j$ for $1\leq j\leq k.$
\begin{thavii}\label{} 
There exist no non-trivial holomorphic solutions for $1/n_1+\cdots+1/n_k< 1/(k-1)$ of  $(\ref{ggg})$  over $S$  satisfying the growth condition
$$\liminf_{r\rightarrow\infty}\frac{\kappa(r)r^2}{\min\{\mathfrak T_{f_{1}}(r), \cdots, \mathfrak T_{f_{k}}(r)\}}=0.$$
 In particular,  the   conclusion holds for $\tilde S=\mathbb C$ without  growth condition.
\end{thavii}

\section{A vanishing theorem for jet differentials}

Let $X$  be a  complex manifold with complex dimension $n.$  A holomorphic $k$-jet differential $\omega$ of  weighted degree $m$  on  $X$ 
is   a homogeneous polynomial 
in $d^i\zeta_j$ $(1\leq i\leq k, 1\leq j\leq n)$  
 of the form
$$\omega=\sum_{|l_1|+\cdots+k|l_k|=m}a_{l_1\cdots l_k}d\zeta^{l_1}\cdots d^k\zeta^{l_k}$$
with holomorphic  function coefficients $a_{l_1\cdots l_k},$  written in  a local holomorphic coordinate $\zeta=(\zeta_1,\cdots,\zeta_n).$ 
Let $D$ be a reduced divisor on $X.$ A logarithmic $k$-jet differential $\omega$ of degree $m$ along $D$ is  a $k$-jet differential of degree $m$ with  possible logarithmic poles along $D,$
i.e., along  $D,$ $\omega$ is locally  a homogeneous polynomial  in
$$d^s\log \sigma_1, \cdots,d^s\log\sigma_r,d^s\sigma_{r+1},\cdots,d^s\sigma_n, \ \ 1\leq s\leq k$$
of weighted degree $m,$ where $\sigma_1,\cdots,\sigma_r$ are irreducible, and $\sigma_1\cdots\sigma_r=0$ is a local defining equation of $D.$

Now let's  introduce a vanishing theorem for jet differentials  shown by the first author \cite{dong}. 
Let  $S$ be an open Riemann surface equipped with a complete  Hermitian metric   such that the  
   Gauss curvature
$K_S\leq0.$
Let $$f:S\rightarrow X$$  be a holomorphic curve into a compact complex  manifold $X.$ 
Let a positive (1,1)-form $\alpha$ on $X.$
 The \emph{Nevanlinna's characteristic} of $f$ with respect to $\alpha$ is defined  by
 \begin{eqnarray*}\label{}
   T_{f,\alpha}(r)
   &=& \pi\int_{D_o(r)}g_r(o,x)f^*\alpha,
 \end{eqnarray*}
 where $D_o(r)$ is the geodesic ball centered at  $o\in S$ with radius $r,$ and $g_r(o,x)$ is the  Green function of $\Delta_S/2$ for $D_o(r)$ with Dirichlet boundary 
 condition, and a pole at $o.$
 The definition for Nevanlinna's characteristic   is  very natural. When $S=\mathbb C,$ the Green function is $(\log\frac{r}{|z|})/\pi,$
 by integration by part,  one can verify  that
it agrees with the classical one.

\begin{theorem}[\cite{dong}]\label{thm2}
 Let $\omega$ be a logarithmic $k$-jet differential on  $X,$ vanishing along an ample divisor  $A$ on $X.$
 Let $f:S\rightarrow X$ be a holomorphic curve
such that $f(S)$ is disjoint from the log-poles of $\omega.$ If $f$  satisfies   
  the growth condition
$$\liminf_{r\rightarrow\infty}\frac{\kappa(r)r^2}{T_{f,A}(r)}=0,$$
where $\kappa$ is defined by $(\ref{kappa}),$ then $f^*\omega\equiv0$ on $S.$ In particular, the conclusion holds when $\tilde S=\mathbb C$ without growth condition.
\end{theorem}

\section{Existence of  solutions of $k$-term Fermat functional equations}
Let $S$ be an open (connected) Riemannn surface. We consider    the  $k$-term Fermat functional equation (\ref{asdf}), i.e., 
$$
f_1^n+\cdots+f^n_k=1
$$
over $S.$   Let 
$\pi: \tilde S\rightarrow S$
be the analytic universal covering of $S.$  A non-trivial holomorphic (resp. meromorphic) solution $(f_1,\cdots,f_k)$ of  (\ref{asdf})  over $S$ can 
 lift to a non-trivial holomorphic (resp. meromorphic) solution 
$(f_1\circ \pi,\cdots,f_k\circ\pi)$ of $(\ref{asdf})$ 
over $\tilde S.$ 
On the other hand, if $(F_1,\cdots,F_k)$ is a non-trivial holomorphic (resp. meromorphic) solution of (\ref{asdf}) over $\tilde S,$  then 
$(F_1\circ\alpha,\cdots,F_k\circ\alpha)$ turns out to be a non-trivial holomorphic (resp. meromorphic) solution of  (\ref{asdf}) over $S$ for 
 a suitable non-constant holomorphic mapping $\alpha: S\rightarrow \tilde S.$ It yields that  
 
\begin{theorem}\label{asd00} Eq. $(\ref{asdf})$ admits a non-trivial holomorphic $($resp. meromorphic$)$ solution over $S$ if and only if Eq. $(\ref{asdf})$ admits a 
non-trivial holomorphic $($resp. meromorphic$)$ solution over $\tilde S.$
\end{theorem}
In what follows, we shall prove by using construction that  there exist non-trivial holomorphic and meromorphic solutions of (\ref{asdf}) when $S$ is hyperbolic. Note from Theorem \ref{asd00} that we only need to handle the case when  $S=\mathbb D.$

\emph{$(i)$ Non-trivial holomorphic solution.}  Let 
$$f_j=a_jz, \ \ j=2,\cdots, k,$$ where $a_2,\cdots a_k$ are  nonzero constants such that  $|a_2^n+\cdots+a_k^n|\leq1.$
Then 
\begin{equation}\label{xxx}
f_1^n=1-a_2^nz^n+\cdots-a_k^nz^n.
\end{equation}
We prove that there is a holomorphic function $f_1$ on $\mathbb D$ satisfying (\ref{xxx}).
Notice that  $\log(1+z)$ is holomorphic on $\mathbb D$ with a Taylor  expansion   
$$\log(1+z)=z-\frac{1}{2}z^2+\frac{1}{3}z^3-\cdots$$
So, $\phi(z):=\log(1-a_2^nz^n-\cdots-a_k^nz^n)$ is holomorphic on $\mathbb D.$ Taking $f_1=e^{\phi/n},$ which is holomorphic  
on $\mathbb D$ and satisfied with (\ref{xxx}). 

\emph{$(ii)$ Non-trivial meromorphic solution.}  We consider three cases:

$a)$  $k=2.$ 
Let $$f_2=az^{-1},$$
where $a$ is a constant such that $|a|\geq1.$ 
It yields  that  $1-f_2^n=(z^n-a^n)/z^n,$ and $\phi_1:=\log(z^n-a^n)$  is holomorphic  
on $\mathbb D.$ Taking $f_1=z^{-1}e^{\phi_1/n},$ which  is meromorphic   on $\mathbb D$ and
    $f_1, f_2$  satisfy  (\ref{asdf}) over $\mathbb D.$

$b)$  $k=3.$ 
Let $$f_2=a_2z^{-1}, \ \ f_3=a_3z^{-1},$$
where $a,b$ are nonzero constants such that $|a_2^n+a_3^n|\geq1.$ Pick $f_1=z^{-1}e^{\phi_2/n},$ where
 $\phi_2:=\log(z^n-a_2^n-a_3^n).$
Then, $(f_1, f_2, f_3)$ is a non-trivial meromorphic solution  satisfying  (\ref{asdf}) over $\mathbb D.$

$c)$  $k\geq4.$ 
Fix a constant $b\not=0.$  Let $$f_2=\sqrt[n]{b}z^{-1}, \ \  f_3=\sqrt[n]{-b}z^{-1}, \ \ f_4=\cdots=f_k=az,$$
where $a\not=0$ is a  constant with $|a\sqrt[n]{k-3}|\leq1.$
So, $\phi_3=\log(1-(k-3)a^nz^n)$ is holomorphic on $\mathbb D.$ Pick $f_1=e^{\phi_3/n},$ which is holomorphic  
on $\mathbb D$ and $f_1,\cdots,f_k$ satisfy (\ref{asdf}).  
We give another non-trivial meromorphic solution as follows
 $$f_1= z^{-1}e^{\frac{\phi_4}{n}}, \  \  f_j=a_jz^{-1}, \ \  j=2,\cdots, k,$$
where $\phi_4:=\log(z^n-a_2^n-\cdots-a_k^n),$ and $a_2,\cdots,a_k$ are nonzero constants such that 
$|a_2^n+\cdots+a_k^n|\geq1.$

According to the above examples  and Theorem \ref{asd00}, we obtain 

\begin{theorem}\label{sss} 
 There exist  non-trivial  holomorphic and meromorphic  solutions of $(\ref{asdf})$ if $S$ is hyperbolic. 
\end{theorem}

We proceed to consider the generalized $k$-term Fermat functional equation 
\begin{equation}\label{fer11}
f_1^{n_1}+\cdots+f_k^{n_k}=1
\end{equation}
over $S.$  

\begin{theorem}\label{asd} Eq. $(\ref{fer11})$ admits a non-trivial holomorphic $($resp. meromorphic$)$ solution over $S$ if and only if Eq. $(\ref{fer11})$  admits a 
non-trivial holomorphic $($resp. meromorphic$)$ solution over $\tilde S.$
\end{theorem}

In what follows, we construct non-trivial holomorphic and meromorphic solutions of (\ref{fer11}) over $\mathbb D.$

\emph{$(i)$ Non-trivial holomorphic solution.}  Set  $n=p_jn_j$ with $1\leq j \leq k,$ where 
$n=[n_1,\cdots,n_k]$ is the lowest common multiple.  Let 
$$f_1= e^{\frac{\psi_1}{n_1}}, \ \  f_j=a_jz^{p_j},  \  \  j=2,\cdots, k,$$
where $\psi_1:=\log(1-a_2^{n_2}z^n-\cdots-a_k^{n_k}z^n),$ and $a_2,\cdots,a_k$ are nonzero constants such that $|a_2^{n_2}+\cdots+a_k^{n_k}|\leq1.$
It is not very difficult to check that $(f_1,\cdots,f_k)$ is a non-trivial holomorphic solution of (\ref{fer11}) over $\mathbb D.$

\emph{$(ii)$ Non-trivial meromorphic solution.}   Let 
$$f_1= z^{-p_1}e^{\frac{\psi_2}{n_1}}, \ \  f_j=a_jz^{-p_j},  \  \  j=2,\cdots, k,$$
where $\psi_2:=\log(z^n-a_2^{n_2}-\cdots-a_k^{n_k}),$ and $a_2,\cdots,a_k$ are nonzero constants such that $|a_2^{n_2}+\cdots+a_k^{n_k}|\geq1.$
We can  check that $(f_1,\cdots,f_k)$ is a non-trivial meromorphic solution of (\ref{fer11}) over $\mathbb D.$ 

Therefore, we conclude that 

\begin{theorem}\label{sss1} 
 There exist  non-trivial  holomorphic and meromorphic  solutions of $(\ref{fer11})$ if $S$ is hyperbolic. 
\end{theorem}

To end this section, we list some examples for the existence of holomorphic and meromorphic solutions of (\ref{asdf}) over $S$ for $k=2,3.$

\noindent\textbf{A. Examples for $k=2$}~

\noindent\textbf{Case $n=2$}

This case is easy. Actually, we can factorize (\ref{asdf})  as  $(f_1+if_2)(f_1-if_2)=1.$ Let  $\alpha=f_1+if_2,$ one can verify that each
 holomorphic (resp. meromorphic)  solution of (\ref{asdf})  over $S$ is of the form 
$$f_1=\frac{\alpha+\alpha^{-1}}{2}, \ \  f_2=\frac{\alpha-\alpha^{-1}}{2i},$$
where $\alpha$ is a holomorphic (resp. meromorphic) function on $S.$ 

\noindent\textbf{Case $n=3$}

Let $\wp$ be the Weierstrass elliptic function on $\mathbb C$  satisfying 
$$(\wp')^2=4\wp^3-1.$$
 Set 
$$\gamma_1=\frac{1}{2\wp}\Big(1-3^{-1/2}\wp'\Big), \ \  \
 \gamma_2=\frac{\varpi}{2\wp}\Big(1+3^{-1/2}\wp'\Big),$$
 where $\varpi$ is a cube-root of unity.  Baker \cite{ba} proved that every meromorphic solution of  (\ref{asdf}) over $\mathbb C$ is of the form
$$ f_1=\gamma_1\circ \alpha, \ \ \ f_2=\gamma_2\circ \alpha,
$$ where $\alpha$ is a holomorphic function on $\mathbb C.$ 
To construct a meromorphic solution of (\ref{asdf}) over $S,$ one  just needs to choose  a holomorphic function $\alpha$ on  $S.$ 
Namely,  
$f_1=\gamma_1\circ  \alpha, f_2=\gamma_2\circ \alpha$
 is   a meromorphic solution of (\ref{asdf}) over $S$ for a  holomorphic function $\alpha$ on $S.$

\noindent\textbf{B. Examples for $k=3$}~

 Let $\alpha, \beta$ be non-constant  holomorphic and meromorphic functions  
  respectively on $S$.  Following constructions of Gundersen-Tohge \cite{gund, gun, gun-toh},  Green \cite{green} and Lehmer \cite{leh}, 
 we have the following examples:

\noindent\textbf{Case $n=2$}

$a)$  \emph{holomorphic solutions}
$$f_1=3^{-1/2}(\alpha^2-2), \ \ f_2=3^{-1/2}(\alpha^2+1)i, \ \ f_3=\sqrt2\alpha.$$

$b)$  \emph{meromorphic solutions}
$$f_1=3^{-1/2}(\beta^2-2), \ \ f_2=3^{-1/2}(\beta^2+1)i, \ \ f_3=\sqrt2\beta.$$

\noindent\textbf{Case $n=3$}

$a)$  \emph{holomorphic solutions}
$$f_1=9\alpha^4, \ \ f_2=-9\alpha^4+3\alpha, \ \ f_3=-9\alpha^3+1.$$

$b)$  \emph{meromorphic solutions}
$$f_1=9\beta^4, \ \ f_2=-9\beta^4+3\beta, \ \ f_3=-9\beta^3+1.$$

\noindent\textbf{Case $n=4$}

$a)$  \emph{holomorphic solutions}
$$
f_1=2^{-3/4}(e^{3\alpha}+e^{-\alpha}), \ \ f_2=(-2)^{-3/4}(e^{3\alpha}-e^{-\alpha}), \ \ 
f_3=(-1)^{1/4}e^{2\alpha}.
$$

$b)$  \emph{meromorphic solutions}

The approach of construction of a holomorphic solution for $n=4$ in above example can be used to construct  meromorphic (non-holomorphic)   solutions for $n=4,$ 
see Green \cite{green}.

\noindent\textbf{Case $n=5$}

$a)$  \emph{holomorphic solutions}
\begin{eqnarray*}
f_1&=&\frac{1}{3}\big((2-\sqrt 6)e^{\alpha}+(2+\sqrt 6)e^{-\alpha}+1\big),  \\
f_2&=&\frac{1}{6}\big((\sqrt 6-2+3\sqrt 2i-2\sqrt 3i)e^{\alpha}-
(\sqrt 6+2-3\sqrt 2i-2\sqrt 3i)e^{-\alpha}
+2\big), \\
f_3&=&\frac{1}{6}\big((\sqrt 6-2+3\sqrt 2i-2\sqrt 3i)e^{\alpha}-
(\sqrt 6+2+3\sqrt 2i+2\sqrt 3i)e^{-\alpha}
+2\big).
\end{eqnarray*}

 $b)$ \emph{meromorphic solutions}
$$f_1=\gamma_1\circ\alpha, \ \ f_2=e^{\pi i/5}\gamma_2\circ\alpha \cdot \gamma_3\circ \alpha, \ \ f_3=\gamma_3\circ\alpha,$$
where $\gamma_1,\gamma_2$ are meromorphic functions on $\mathbb C$ given by 
$$\gamma_1=1+\frac{1}{p_1+p_2e^z}, \ \ \gamma_2=1+\frac{1}{p_1+p_2e^{-z}}$$
with $$p_1=\frac{a_3a_4-a_1a_2}{a_3+a_4-a_1-a_2}, \ \ p_2=\sqrt{(p_1-a_1)(p_1-a_2)}, \ \ a_k=\frac{1}{e^{2k\pi i/5}-1},$$
and
$\gamma_3$ is a meromorphic function  on $\mathbb C$  satisfying 
$$\gamma_3^5=\frac{\gamma_1^5-1}{\gamma_2^5-1}.$$

\noindent\textbf{Cases $n=6, 7, 8$}

A  non-holomorphic meromorphic solution of (\ref{asdf}) for $n=6$ exists by using the  construction  of Gundersen 
over $\mathbb C.$  Gundersen \cite{gund} proved that  there exist a 
non-holomorphic meromorphic  solution $(F_1, F_2, F_3)$ of (\ref{asdf}) over $\mathbb C.$
Now let $f_1=F_1\circ\alpha, f_2=F_2\circ\alpha, f_3=F_3\circ\alpha,$ then $f_1,f_2,f_3$ solve (\ref{asdf}) over $S.$

It is unknown that  whether there exists a non-trivial holomorphic solution of (\ref{asdf}) for $n=6,$ and a
 non-trivial meromorphic (non-holomorphic) solution of (\ref{asdf})  for $n=7,8$
 for a general  open Riemann surface $S.$
However, it  always exists non-trivial solutions  if $S$ is hyperbolic since Theorem \ref{sss}.

\section{Non-existence of  solutions of 2-term and 3-term Fermat functional equations}

In this section,  we employ a vanishing theorem for jet differentials to study the  non-existence of  non-trivial holomorphic and meromorphic solutions of 2-term and 3-term Fermat functional equations over open Riemann surfaces.
It  is mentioned  that the jet differential approach was  used by Ng-Yeung \cite{ny}.

\subsection{2-term Fermat functional equations}~

Let  $\omega_{FS}$ be the Fubini-Study form on $\mathbb P^2(\mathbb C)$ with  homogeneous coordinate  $[X: Y: Z].$  
Set  $x=X/Z, y=Y/Z.$ We 
treat the 2-term Fermat functional equation
\begin{equation}\label{2term}
f^n+g^n=1
\end{equation}
over $S.$   A Fermat curve of degree $n$ in $\mathbb P^2(\mathbb C)$ is defined by
$$
C_n: \ X^n+Y^n=Z^n, 
$$
which is a compact Riemann surface of genus $(n-1)(n-2)/2$ and  its   affine form  is written as 
\begin{equation}\label{fermat10}
x^n+y^n=1.
\end{equation}
A  holomorphic or meromorphic solution $(f,g)$ of (\ref{2term})  is viewed as a holomorphic curve $\mathfrak C: S\rightarrow C_n \subset\mathbb P^2(\mathbb C)$  by 
$$x\mapsto [f(x): g(x): 1], \ \  \ ^\forall x\in S.$$
Locally, write $\mathfrak C$   in a holomorphic form 
$\mathfrak C=[\hat f: \hat g: e]$
 with 
 $$f=\hat f/e, \ \  g=\hat g/e.$$ The Nevanlinna's characteristic  of $\mathfrak C$ with respect to $\omega_{FS}$ is defined by 
 \begin{equation*}
T_{\mathfrak C,\omega_{FS}}(r)=\frac{1}{4}\int_{D_o(r)}g_r(o,x)\Delta_S\log( |\hat f(x)|^2+|\hat g(x)|^2+|e(x)|)dV(x).
\end{equation*}
 Differentiating (\ref{fermat10}) to get  
\begin{eqnarray*}
x^{n-1}dx+y^{n-1}dy &=&0.
\end{eqnarray*}
Apply Crammer's rule to this equation and (\ref{fermat10}),  it yields  that 
\begin{equation}\label{aaas}
\Phi:=\frac{dy}{x^{n-1}}
=\frac{-dx}{y^{n-1}}
=
\left |\begin{array}{cccc}
x & y \\
dx & dy  
\end{array}\right|
\end{equation}
which is a  1-jet differential on $C_n.$ We prove that $\Phi$ is holomorphic for $n\geq 3$ and vanishing along $\infty$ for $n\geq 4,$ where $\infty:=(Z=0)\cap C_n$ denotes an ample divisor on  $C_n.$  
From the last term of (\ref{aaas}), one verifies that $\Phi$ is holomorphic on the affine part  $\{Z\not=0\}\cap C_n.$
Now, let us look at the pole order of $\Phi$ at  $\infty.$ The numerator $dy$ in the first term of (\ref{aaas}) 
 gives a pole of order 2 at $\infty,$ 
 and the corresponding denominator $x^{n-1}$ gives rise to a zero of order $n-1$ of $\Phi$  at 
$\infty.$
 Hence, $\Phi$  has a pole of order $3-n$ at $\infty.$  This implies that  
 $\Phi$ is holomorphic when $n\geq 3$ and vanishing along $\infty$ when  $n\geq 4.$

Define
$$\mathfrak T_{f,g}(r):=\frac{1}{4}\int_{D_o(r)}g_r(o,x)\Delta_S\log(1+ |f(x)|^2+|g(x)|^2)dV(x).$$
Clearly, 
 $$\mathfrak T_{f,g}(r)\leq T_{\mathfrak C, \omega_{FS}}(r).$$ 
\begin{theorem}\label{t1} For  $n\geq 4,$ there exist no non-trivial meromorphic solutions of  $(\ref{2term})$  satisfying the growth condition
$$\liminf_{r\rightarrow\infty}\frac{\kappa(r)r^2}{\mathfrak T_{f,g}(r)}=0,$$
where $\kappa$ is defined by $(\ref{kappa}).$  In particular,   there are no non-trivial meromorphic solutions  for  $n\geq 4$ if  $\tilde S=\mathbb C.$
\end{theorem}
\begin{proof}
When  $n\geq 4,$  $\Phi$  (given by (\ref{aaas})) is a holomorphic 1-jet differential on $C_n$ which vanishes along  $\infty.$ 
The growth condition leads to
$$\liminf_{r\rightarrow\infty}\frac{\kappa(r)r^2}{T_{\mathfrak C,\omega_{FS}}(r)}=0$$
since $\mathfrak T_{f,g}(r)\leq T_{\mathfrak C, \omega_{FS}}(r).$ 
Invoking Theorem \ref{thm2}, we obtain $\mathfrak C^*\Phi\equiv0.$ Hence,  
 $\mathfrak C$ satisfies the differential equation 
$$xdy-ydx=0$$
which is solved generally by  
$y=ax,$
where $a$ is an  arbitrary constant. Hence, we obtain  
$g=af.$ Combine this with (\ref{2term}), 
we can prove  the theorem.
\end{proof}

Now, we consider  holomorphic solutions of  (\ref{2term}). Rewrite (\ref{2term}) as the form
\begin{equation}\label{2term1}
F^n+1=G^n
\end{equation}
with 
$$F=f/g, \ \   G=1/g.$$
Accordingly,  (\ref{fermat10}) is written as
\begin{equation}\label{fermat21}
u^n+1=v^n
\end{equation}
with 
$$u=x/y=X/Y, \ \   v=1/y=Z/Y.$$
Differentiating (\ref{fermat21}) to get  
\begin{eqnarray*}
u^{n-1}du-v^{n-1}dv&=&0,
\end{eqnarray*}
It yields from  Crammer's rule that 
\begin{equation}\label{aaa}
\Psi:=\frac{dv}{u^{n-1}}
=\frac{du}{v^{n-1}}
=
\left |\begin{array}{cccc}
v & u \\
dv & du  
\end{array}\right|,
\end{equation}
which is a 1-jet differetial on $C_n.$  
Set 
\begin{equation}\label{eta1}
\eta=\frac{1}{v}\Psi,
\end{equation}
 which is a logarithmic 1-jet differential with log-poles along $v=0$ for $n\geq 2,$ and vanishing along the ample divisor $(Y=0)\cap C_n$  for $n\geq 3.$ The argument is standard and similar to  before.

\begin{theorem}\label{t2} For  $n\geq 3,$ there exists no non-trivial holomorphic solutions of  $(\ref{2term})$  satisfying the growth condition
$$\liminf_{r\rightarrow\infty}\frac{\kappa(r)r^2}{\mathfrak T_{f,g}(r)}=0,$$
where $\kappa$ is defined by $(\ref{kappa}).$  In particular,  there are no non-trivial meromorphic solutions  for  $n\geq 3$ if  $\tilde S=\mathbb C.$
\end{theorem}
\begin{proof}  Let $(f,g)$ be a  holomorphic solution of (\ref{2term}).  Note that  $G=1/g$ omits the value $0,$ this means that $\mathfrak C(S)$ avoids the log-poles of $\eta$ given by (\ref{eta1}) for $n\geq 2.$ Moreover, $\eta$ vanishes along the ample divisor $(Y=0)\cap C_n$  for $n\geq 3.$ 
By Theorem \ref{thm2}, we have   $\mathfrak C^*\eta\equiv0$ under the assumed  growth condition, 
i.e.,   $\mathfrak C^*v^{-1}\equiv0$ or  $\mathfrak C$ satisfies the differential equation 
$$\left |\begin{array}{cccc}
u & v \\
 du &  dv  
\end{array}\right|=0.$$
The first  case is  handled trivially. For the else case, we obtain   
$v=au,$
where $a$ is an  arbitrary constant. This implies that $f,g$ are constants. 
\end{proof}

\subsection{3-term Fermat functional equations}~

Let  $\omega_{FS}$ be the Fubini-Study form on $\mathbb P^3(\mathbb C)$ with  homogeneous coordinate  $[X: Y: Z: W].$  
Set  $x=X/W, y=Y/W, z=Z/W.$ 
We consider the 3-term Fermat functional equation
\begin{equation}\label{3term}
f^n+g^n+h^n=1
\end{equation}
over $S.$ A Fermat surface of degree $n$ in $\mathbb P^3(\mathbb C)$ is defined by
$$
S_n: \ X^n+Y^n+Z^n=W^n,
$$
which is a smooth complex surface and its affine form  is written as 
\begin{equation}\label{fermat1}
x^n+y^n+z^n=1.
\end{equation}
A  holomorphic or meromorphic solution $(f,g,h)$ of (\ref{3term}) can be regarded  as a holomorphic curve $\mathfrak C: S\rightarrow S_n \subset\mathbb P^3(\mathbb C)$  by 
$$x\mapsto [f(x): g(x): h(x): 1], \ \  \ ^\forall x\in S.$$
Locally, write $\mathfrak C$   in a holomorphic form 
$\mathfrak C=[\hat f: \hat g: \hat h: e]$
 with 
  $$f=\hat f/e, \ \  g=\hat g/e, \ \  h=\hat h/e.$$
   The Nevanlinna's characteristic  of $\mathfrak C$  with respect to $\omega_{FS}$ is defined by 
 \begin{equation*}
T_{\mathfrak C,\omega_{FS}}(r)=\frac{1}{4}\int_{D_o(r)}g_r(o,x)\Delta_S\log( |\hat f(x)|^2+|\hat g(x)|^2+|\hat h(x)|^2+|e(x)|^2)dV(x).
\end{equation*}
Differentiating (\ref{fermat1}) to get  
\begin{eqnarray*}
x^{n-1}dx+y^{n-1}dy+z^{n-1}dz&=&0, \\
x^{n-1}\mathscr D^2x+y^{n-1}\mathscr D^2y+z^{n-1}\mathscr D^2z&=&0,
\end{eqnarray*}
where $$\mathscr D^2\psi:=d^2\psi+(n-1)d\psi^2/\psi$$ for a function $\psi.$ Apply Crammer's rule to the  two equations as well as (\ref{fermat1}),  
\begin{equation}\label{aa}
\Phi:=\frac{\left |\begin{array}{cccc}
dy & dz \\
\mathscr D^2 y & \mathscr D^2z  \\
\end{array}\right|}{x^{n-1}}
=\frac{\left |\begin{array}{cccc}
dz & dx \\
\mathscr D^2 z & \mathscr D^2x  \\
\end{array}\right|}{y^{n-1}}
=\frac{\left |\begin{array}{cccc}
dx & dy \\
\mathscr D^2 x & \mathscr D^2y  \\
\end{array}\right|}{z^{n-1}}=
\left |\begin{array}{cccc}
x & y & z \\
dx & dy & dz \\
\mathscr D^2 x & \mathscr D^2 y & \mathscr D^2z 
\end{array}\right|
\end{equation}
which is a 2-jet differential on $S_n.$ 
Set 
\begin{equation}\label{w}
\omega=xyz\Phi,
\end{equation}
 which is holomorphic when $n\geq 8$ and vanishing along $\infty$ when
 $n\geq 9,$ where $\infty:=(W=0)\cap S_n$ is an ample divisor on  $S_n.$  The argument states  as follows.
Observing that $\mathscr D^2x, \mathscr D^2y, \mathscr D^2z$  are only of simple poles at $x=0, y=0, z=0$ respectively, hence $\omega$ is holomorphic on the affine part $\{W\not=0\}\cap S_n$ 
due to the last term of (\ref{aa}). Next, we look at the pole order of $\Phi$ at  $\infty.$ Expanding $\Phi$ (the third term in (\ref{aa})) to get 
\begin{eqnarray*}
\Phi=\frac{dxd^2y-dyd^2x+(n-1)(d\log y-d\log x)dxdy}{z^{n-1}}.
\end{eqnarray*}
By a direct computation,  we  obtain  $$dxd^2y-dyd^2x=d\big(\frac{dy}{dx}\big)dx^2$$
which  has a pole of order 4 at $\infty.$ Moreover, the denominator $z^{n-1}$ gives rise to a zero of order $n-1$ of $\Phi$  at 
$\infty.$
 Therefore, $\omega$  has a pole of order $8-n$ at $\infty.$  So, 
 $\omega$ is holomorphic for $n\geq 8$ and vanishing along $\infty$ for  $n\geq 9.$

Define
$$\mathfrak T_{f,g,h}(r):=\frac{1}{4}\int_{D_o(r)}g_r(o,x)\Delta_S\log(1+ |f(x)|^2+|g(x)|^2+|h(x)|^2)dV(x).$$
Clearly, 
 $$ \mathfrak T_{f,g,h}(r)\leq T_{\mathfrak C, \omega_{FS}}(r).$$ 
\begin{theorem}\label{t3} For  $n\geq 9,$ there exist  no non-trivial meromorphic solutions of  $(\ref{3term})$  satisfying the growth condition
$$\liminf_{r\rightarrow\infty}\frac{\kappa(r)r^2}{\mathfrak T_{f,g,h}(r)}=0,$$
where $\kappa$ is defined by $(\ref{kappa}).$  In particular,   there are no non-trivial meromorphic solutions  for  $n\geq 9$ if  $\tilde S=\mathbb C.$
\end{theorem}
\begin{proof}
If  $n\geq 9,$ then $\omega$  (defined by (\ref{w})) is a holomorphic 2-jet differential on $S_n$ vanishing along  $\infty.$ 
The growth condition implies that 
$$\liminf_{r\rightarrow\infty}\frac{\kappa(r)r^2}{T_{\mathfrak C,\omega_{FS}}(r)}=0$$
due to $ \mathfrak T_{f,g,h}(r)\leq T_{\mathfrak C, \omega_{FS}}(r).$ 
By Theorem \ref{thm2}, it yields that $\mathfrak C^*\omega\equiv0.$ Thus,  
 $\mathfrak C^*x\equiv0$ or $\mathfrak C^* y\equiv0$ or $\mathfrak C^*z\equiv0;$ or else
 $\mathfrak C$ satisfies the differential equation
\begin{equation}\label{hhh1}
\left |\begin{array}{cccc}
dx & dy \\
\mathscr D^2 x & \mathscr D^2y  \\
\end{array}\right|=0.
\end{equation}
\ \ \ \ 
If $\mathfrak C^*x\equiv0,$  then  (\ref{3term}) reduces to $g^n+h^n=1.$ Invoking Theorem \ref{t1},  
there exists no  non-trivial holomorphic solution, and so does (\ref{3term}). 
The  arguments are applicable to both  cases  
$\mathfrak C^* y\equiv0$ and $\mathfrak C^*z\equiv0.$
For the  else case,  it yields from  (\ref{hhh1}) that 
$$d\big(\frac{dy}{dx}\big)dx^2+(n-1)\big(d\log\frac{y}{x}\big)dxdy=0$$
which is solved generally by  
$y^n=ax^n+b,$
where $a,b$ are  arbitrary constants. Hence,  we conclude that  
\begin{equation}\label{w11}
g^n-af^n=b.
\end{equation}
If $ab\not=0,$ then  (\ref{w11}) has no non-trivial holomorphic solutions due to Theorem \ref{t1}, and so does (\ref{3term}). If $a=0,$  then 
  $f^n+h^n=1-b.$ 
 Invoking Theorem \ref{t1} again, then we  also  verify  that (\ref{3term}) has no non-trivial holomorphic solutions.  
If  $b=0,$ then $(1+a)f^n+h^n=1.$ The similar argument will verify this case. We conclude the proof.
\end{proof}

Now, we consider  holomorphic solutions of (\ref{3term}).  Rewrite (\ref{3term}) as the form
\begin{equation}\label{3term1}
F^n+G^n+1=H^n
\end{equation}
with 
$$F=f/h, \ \ G=g/h, \ \ H=1/h.$$
Accordingly,  (\ref{fermat1}) is written as
\begin{equation}\label{fermat2}
u^n+v^n+1=w^n
\end{equation}
with 
$$u=x/z=X/Z, \ \  v=y/z=Y/Z,  \ \  w=1/z=W/Z.$$
Differentiating (\ref{fermat2}) to get  
\begin{eqnarray*}
u^{n-1}du+v^{n-1}dv-w^{n-1}dw&=&0, \\
u^{n-1}\mathscr D^2u+v^{n-1}\mathscr D^2v-w^{n-1}\mathscr D^2w&=&0.
\end{eqnarray*}
Apply  Crammer's rule  to the two equations as well as (\ref{fermat2}), 
$$
\Psi:=\frac{\left |\begin{array}{cccc}
dw & dv \\
\mathscr D^2 w & \mathscr D^2v  \\
\end{array}\right|}{u^{n-1}}
=\frac{\left |\begin{array}{cccc}
dw & du \\
\mathscr D^2 w & \mathscr D^2u  \\
\end{array}\right|}{v^{n-1}}
=\frac{\left |\begin{array}{cccc}
du & dv \\
\mathscr D^2 u & \mathscr D^2v  \\
\end{array}\right|}{w^{n-1}}=
\left |\begin{array}{cccc}
u & v & w \\
du & dv & dw \\
\mathscr D^2 u & \mathscr D^2 v & \mathscr D^2w 
\end{array}\right|
$$ 
which is a 2-jet differential on $S_n.$ 
Set 
\begin{equation}\label{eta}
\eta=\frac{uv}{w}\Psi,
\end{equation}
 which is a holomorphic logarithmic 2-jet differential with only  log-poles along $w=0$ when $n\geq 6,$ and 
 vanishing along the ample divisor $(Z=0)\cap S_n$  when $n\geq 7.$ The argument is standard and similar to before. 

\begin{theorem}\label{t4} For  $n\geq 7,$ there exist no non-trivial holomorphic solutions of  $(\ref{3term})$  satisfying the growth condition
$$\liminf_{r\rightarrow\infty}\frac{\kappa(r)r^2}{\mathfrak T_{f,g,h}(r)}=0,$$
where $\kappa$ is defined by $(\ref{kappa}).$  In particular,   there are no non-trivial holomorphic solutions  for  $n\geq 7$ if  $\tilde S=\mathbb C.$
\end{theorem}
\begin{proof}  Let $(f,g,h)$ be a  holomorphic solution of (\ref{3term}).  Then, $H=1/h$ omits the value $0,$ this means that $\mathfrak C(S)$ avoids the log-poles of $\eta$ given by (\ref{eta}) for $n\geq 6.$ Indeed, $\eta$ vanishes along the ample divisor $(Z=0)\cap S_n$  for $n\geq 7.$ 
By Theorem \ref{thm2}, this  follows that $\mathfrak C^*\eta\equiv0$ due to the assumed  growth condition. 
Hence,  $\mathfrak C^*u\equiv0$ or $\mathfrak C^* v\equiv0$ or $\mathfrak C^*w^{-1}\equiv0;$ or else $\mathfrak C$ satisfies the differential equation 
$$\left |\begin{array}{cccc}
du & dv \\
\mathscr D^2 u & \mathscr D^2v  \\
\end{array}\right|=0.$$
The  first three cases are trivially handled. In the last case,  $\mathfrak C$ satisfies 
$$dud^2v-dvd^2u+(n-1)(d\log v-d\log u)dudv=0
$$which is solved generally by  
$v^n=au^n+b,$
where $a,b$ are  arbitrary constants. Whence, we obtain 
$
af^n+bh^n=g^n.
$
Substituting this equation into  (\ref{3term}), we get   $(1+a)f^n+(1+b)h^n=1.$  Clearly,  (\ref{3term}) exists  at most trivial holomorphic solution for $a=-1$ or $b=-1.$ 
If $a\not=-1$ and $b\not=-1,$ then there still exists  no non-trivial holomorphic solution of  (\ref{3term}) since Theorem \ref{t2}. The proof is completed.
\end{proof}

\section{Generalized Fermat functional equations}

We investigate  holomorphic and meromorphic solutions of the   generalized 2-term and 3-term Fermat functional equations
\begin{equation}\label{2term10} 
f^m+g^n=1;
\end{equation} 
\begin{equation}\label{3term10} 
f^m+g^n+h^l=1
\end{equation}
 over an  open Riemann surface $S.$ Treat  the Fermat curve $C_{m,n}$ and Fermat surface $S_{m,n,l}$ defined  by 
 $$ C_{m,n}: \ X^m+Y^nZ^{m-n}=Z^m, \ \ m\geq n;$$
 $$S_{m,n,l}: \ X^m+Y^nW^{m-n}+Z^lW^{m-l}=W^m, \ \ m\geq n\geq l $$
 in $\mathbb P^2(\mathbb C)$ and $\mathbb P^3(\mathbb C)$ respectively, 
their  affine forms are written  as
 $$ x^m+y^n=1; \ \ \  x^m+y^n+z^l=1$$
respectively. Then there are  holomorphic curves 
$$\mathfrak C_1=[f:g:1]: \ S\rightarrow C_{m,n}\subset \mathbb P^2(\mathbb C); $$
 $$\mathfrak C_2=[f:g:h:1]: \ S\rightarrow S_{m,n,l}\subset \mathbb P^3(\mathbb C).$$

\subsection{Non-existence of meromorphic solutions}~

Firstly, we assume that $C_{m,n}, S_{m,n,l}$ are normal,  with only  possible isolated singularities. It will be discussed in two cases as follows.

$(i)$  $C_{m,n}, S_{m,n,l}$\emph{ have no singularities}

Differentiating  $x^m+y^n=1$ to get 
$$mx^{m-1}dx+ny^{n-1}dy=0,$$ which gives rise to  a  1-jet differential 
$$\Phi_1:=\frac{dy}{mx^{m-1}}
=\frac{-dx}{ny^{n-1}}
=
\left |\begin{array}{cccc}
x & y \\
dx & dy  
\end{array}\right|
$$
on $C_{m,n}.$ Like before,  we can show that $\Phi_1$ is holomorphic when $1/m+1/n\leq 2/3$ and vanishing along  $(Z=0)\cap C_{m,n}$  when $1/m+1/n\leq 1/2.$   
This  because that the condition  $1/m+1/n\leq 2/3$ ensures that $m\geq 3$ or $n\geq 3,$ and the condition  $1/m+1/n\leq 1/2$  ensures that $m\geq 4$ or $n\geq4.$

Differentiating $x^m+y^n+z^l=1$ to get  
\begin{eqnarray*}
mx^{m-1}dx+ny^{n-1}dy+lz^{l-1}dz&=&0, \\
mx^{m-1}\mathscr D_x^2x+ny^{n-1}\mathscr D_y^2v+lx^{l-1}\mathscr D_z^2z&=&0,
\end{eqnarray*}
where 
$$\mathscr D_x^2x=d^2x+\frac{m-1}{x}dx^2, \ 
\mathscr D_y^2y=d^2y+\frac{n-1}{y}dy^2, \ 
\mathscr D_z^2z=d^2z+\frac{l-1}{z}dz^2.
$$
It gives a  2-jet differential 
$$
\Phi_2:=\frac{\left |\begin{array}{cccc}
dy & dz \\
\mathscr D_y^2 y & \mathscr D_z^2z  \\
\end{array}\right|}{mx^{m-1}}
=\frac{\left |\begin{array}{cccc}
dz & dx \\
\mathscr D_z^2 z & \mathscr D_x^2x  \\
\end{array}\right|}{ny^{n-1}}
=\frac{\left |\begin{array}{cccc}
dx & dy \\
\mathscr D_x^2 x & \mathscr D_y^2y  \\
\end{array}\right|}{lz^{l-1}}=
\left |\begin{array}{cccc}
x & y & z \\
dx & dy & dz \\
\mathscr D_x^2 x & \mathscr D_y^2 y & \mathscr D_z^2z 
\end{array}\right|
$$ 
on $S_{m,n,l}.$ Set $\omega=xyz\Phi_2,$ 
 which is holomorphic when $1/m+1/n+1/l\leq 3/8$ and 
 vanishing along  $(W=0)\cap S_{m,n,l}$  when $1/m+1/n+1/l \leq 1/3.$ 
It  because that the condition  $1/m+1/n+1/l\leq 3/8$ ensures that $m\geq 8$ or $n\geq 8$ or $l\geq 8,$ and the condition  $1/m+1/n+1/l \leq 1/3$ 
ensures that $m\geq 9$ or $n\geq9$ or $l\geq 9.$

$(ii)$ $C_{m,n}, S_{m,n,l}$ \emph{have only isolated singularities}

It's  very trivial to check that   $C_{m,n}$ has a unique  singularity $[0:1:0]$  lying in   the ample  divisor  $(Z=0)\cap C_{m,n},$ hence $\Phi_1$ is holomorphic on the affine 
part  $\{Z\not=0\}\cap C_{m,n}.$ With the similar arguments as in case $(i),$  $\Phi_1$  vanishes
 along  $(Z=0)\cap C_{m,n}$  for $1/m+1/n\leq 1/2.$   Now, one looks at $S_{m,n,l}.$  
Let $\mathfrak S$ be the set of  singularities of $S_{m,n,l},$ then $\omega$ is holomorphic on   the affine part $\{W\not=0\}\cap S_{m,n,l}$ outside $\mathfrak S.$
If a singularity $P\in \mathfrak S$ lies in $\{W \not=0\},$ one can check that $\omega$ is bounded near $P.$ 
Therefore, $\omega$ can extend across $P$ since the normality of $S_{m,n,l}.$
Let  $\pi: \tilde S_{m,n,l}\rightarrow S_{m,n,l}$  be the resolution of $\mathfrak S,$ 
then $\pi^*\omega$   
is  holomorphic  on  $\pi^*\{W\not=0\}\cap \tilde S_{m,n,l}.$ Similarly, along  $\pi^*(W=0)\cap \tilde S_{m,n,l},$ $\pi^*\omega$   is vanishing 
  when $1/m+1/n+1/l \leq 1/3.$ Morevover, one can lift $\mathfrak C_2$ to $\tilde{\mathfrak C}_2: S\rightarrow \tilde S_{m,n,l}.$
This turns to case $(i)$  when $1/m+1/n+1/l \leq 1/3.$

In what follows, we treat $S_{m,n,l}$  ($C_{m,n}$ can be  handled trivially) in a  general case.   
Note first that $S_{m,n,l}$ is a  Delsarte surface and which has degree $\geq9$  for $1/m+1/n+1/l \leq 1/3.$   
By Heijne \cite{hei}, 
$S_{m,n,l}$ belongs to one of  83 classes of Delsarte surfaces with only isolated ADE singularities
up to an isomorphism. Therefore, $S_{m,n,l}$ is a normal surface 
with only isolated ADE singularities for $1/m+1/n+1/l \leq 1/3.$ This  turns to  case $(ii)$ when $1/m+1/n+1/l \leq 1/3.$

According to the above discussions,   the similar arguments  as in the proofs of Theorem \ref{t1} and Theorem \ref{t3} follow immediately that 
\begin{theorem}\label{} For $1/m+1/n\leq 1/2,$ there exist no non-trivial meromorphic solutions of  $(\ref{2term10})$  satisfying the growth condition
$$\liminf_{r\rightarrow\infty}\frac{\kappa(r)r^2}{\mathfrak T_{f,g}(r)}=0,$$
where $\kappa$ is defined by $(\ref{kappa}).$  In particular,  there are no non-trivial meromorphic solutions  for  $1/m+1/n\leq 1/2$ if  $\tilde S=\mathbb C.$
\end{theorem}

\begin{theorem}\label{} For $1/m+1/n+1/l\leq 1/3,$ there exist no non-trivial meromorphic solutions of  $(\ref{3term10})$  satisfying the growth condition
$$\liminf_{r\rightarrow\infty}\frac{\kappa(r)r^2}{\mathfrak T_{f,g,h}(r)}=0,$$
where $\kappa$ is defined by $(\ref{kappa}).$  In particular,  there are no non-trivial meromorphic solutions  for  $1/m+1/n+1/l \leq 1/3$ if  $\tilde S=\mathbb C.$
\end{theorem}

\subsection{Non-existence of holomorphic solutions}~

Let $\psi$ be a non-constant meromorphic function on $S.$ Following Dong \cite{dong}, we define the Nevanlinna's functions  of $\psi$ over $S$ by 
 \begin{eqnarray*}
  N(r,\psi)
&=& \pi \sum_{\psi*\infty\cap D_o(r)}g_r(o,x),  \\
m(r,\psi)&=& \int_{\partial D_o(r)} \log^+|\psi(x)|d\pi_r^o(x), \\
 T(r,\psi)
   &=& \frac{1}{4}\int_{D_o(r)}g_r(o,x)\Delta_S\log(1+|\psi(x)|^2)dV(x).
 \end{eqnarray*}
For $a\in \mathbb P^1(\mathbb C),$  we have
$$\text{F.  M. T.}  \ \ \  T(r, \psi)=m\big(r,1/(\psi-a)\big)+N\big(r,1/(\psi-a)\big)+O(1).$$
Similarly, one can define the $k$-truncated counting function $N^{[k]}(r,\psi)$ in such manner: if $x_0$ is a pole  of $\psi$ with multiplicity $\mu$ in $D_o(r),$
then one   just 
 takes $x_0$ $\min\{\mu, k\}$
times, namely, one only keeps the part $\pi\min\{\mu, k\}g_r(o,x_0)$  for $x_0$ in the expression of $N(r,\psi).$ 
The $k$-level defect of $\psi$ with respect to $a$ is defined by
$$\delta^{[k]}(\psi, a)=1-\limsup_{r\rightarrow\infty}\frac{N^{[k]}\big(r,1/(\psi-a)\big)}{T(r,\psi)}.$$
In short,  write  $\delta(\psi, a):=\delta^{[\infty]}(\psi, a).$ Clearly, we have
$$0\leq \delta(\psi, a) \leq \delta^{[k]}(\psi, a)\leq1.$$
\ \ \ \ Since $S$ is open, then there exists  a nowhere-vanishing holomorphic  vector field $\mathfrak X$ over $S,$ see \cite{gr}. 
Let $\psi_0, \cdots, \psi_n$ $(n\geq 1)$ be non-constant holomorphic functions on $S,$ 
   define the Wronskian determinant of $\psi_0,\cdots,\psi_n$ with respect to $\mathfrak{X}$   by
$$W_\mathfrak{X}(\psi_0,\cdots,\psi_n)=\left|
  \begin{array}{ccc}
   \psi_0 & \cdots & \psi_n \\
   \mathfrak{X}(\psi_0) & \cdots & \mathfrak{X}(\psi_n) \\
    \vdots & \vdots & \vdots \\
     \mathfrak{X}^n(\psi_0) & \cdots & \mathfrak{X}^n(\psi_n) \\
  \end{array}
\right|.$$

We introduce a Logarithmic Derivative Lemma as follows
\begin{lemma}[\cite{dong}]\label{ldl} Let $\psi$ be a non-constant meromorphic function on $S.$ For  a positive integer $k,$ we have 
\begin{eqnarray*}
m\Big(r,\frac{\mathfrak{X}^k(\psi)}{\psi}\Big)  &\leq_{\rm{exc}}& \frac{3k}{2}\log T(r,\psi)+O\Big(\log^+\log T(r,\psi)-\kappa(r)r^2
+\log^+\log r\Big)
\end{eqnarray*}
with $\mathfrak X^j=\mathfrak X\circ\mathfrak X^{j-1}$ and $\mathfrak X^0=id,$   where
$\kappa$ is defined by $(\ref{kappa}).$
\end{lemma}

\begin{lemma}\label{lem5}  Let $\psi_0, \cdots, \psi_n$ $(n\geq 1)$ be non-constant holomorphic functions on $S$ satisfying 
$$a_0\psi_0+\cdots+a_n\psi_n=0.$$
 If 
$$\liminf_{r\rightarrow\infty}\frac{\kappa(r)r^2}{\min\{T(r, \psi_0),\cdots,T(r,\psi_n)\}}=0,$$
where $\kappa$ is defined by $(\ref{kappa}),$ then
$$\sum_{j=0}^n\delta^{[n]}(\psi_j, 0)\leq n.$$
\end{lemma}
\begin{proof}  We  prove the lemma by considering  two cases.

$a)$ $\psi_0,\cdots,\psi_n$ are linearly independent over $\mathbb C$

 Differentiating
$\psi_0+\cdots+\psi_n=0$ to get 
$$\sum_{j=0}^n\frac{\mathfrak X^\mu\psi_j}{\psi_j}\psi_j=0, \ \ \mu=1,\cdots, n.$$
It yields from Crammer's rule  that  
$\psi_j=\Delta_j/\Delta,$
where 
$$\Delta=\frac{W_\mathfrak{X}(\psi_0,\cdots,\psi_n)}{\psi_0\cdots\psi_n}, \ \ \ 
\Delta_j=\frac{\psi_jW_\mathfrak{X}(\psi_0,\cdots,\psi_{j-1},1,\psi_{j+1},\cdots, \psi_n)}{\psi_0\cdots\psi_n}.$$
The First Main Theorem and Lemma \ref{ldl} imply that 
 \begin{eqnarray*}\label{}
m(r,\psi_j)&\leq& m(r,\Delta_j)+m(r,1/\Delta)+O(1) \\
&\leq& m(r,\Delta_j)+m(r,\Delta)+N(r,\Delta)+O(1) \\
&\leq& N(r,\Delta)+S(r) \\
&\leq& \sum_{j=0}^nN^{[n]}(r,1/\psi_j)+S(r),
\end{eqnarray*}
where 
$$S(r)=O\Big(\sum_{j=0}^n\log T(r,\psi_j)-\kappa(r)r^2\Big).$$
Therefore, 
$$T(r):=\max\big\{T(r,\psi_0),\cdots, T(r,\psi_n)\big\} 
\leq \sum_{j=0}^nN^{[n]}(r,1/\psi_j)+S(r).
$$
For an arbitrary $\epsilon>0,$
$$N^{[n]}(r,1/\psi_j)\leq\big(1-\delta^{[n]}(\psi_j, 0)+\epsilon\big)T(r,\psi_j)+S(r)$$
holds for $r$  large enough.  Thus, it follows  that 
$$T(r)\leq \sum_{j=0}^n\big(1-\delta^{[n]}(\psi_j, 0)+\epsilon\big)T(r)+S(r).$$
This implies that 
$\delta^{[n]}(\psi_0, 0)+\cdots+\delta^{[n]}(\psi_n, 0)\leq n.$

$b)$  $\psi_0,\cdots,\psi_n$ are linearly dependent over $\mathbb C$

Rewrite $\psi_0+\cdots+\psi_n$ as 
$a_0\psi_{n_0}+\cdots+a_k\psi_{n_k}$
such that $\psi_{n_0}, \cdots, \psi_{n_k}$ are  linearly independent over $\mathbb C.$
By $a),$ it yields that   
$$T_1(r)\leq \sum_{j=0}^k\big(1-\delta^{[k]}(\psi_{n_j}, 0)+\epsilon\big)T_1(r)+S(r),$$
where 
$$T_1(r)=\max\big\{T(r,\psi_{n_0}),\cdots, T(r,\psi_{n_k})\big\}.$$
Then
$$\sum_{j=0}^k\delta^{[k]}(\psi_{n_j}, 0)\leq k.$$
Notice that $\delta^{[n]}(\psi_{j}, 0)\leq \delta^{[k]}(\psi_{j}, 0)\leq1,$   we  confirm the lemma.
\end{proof}

\begin{theorem}\label{} For $1/m+1/n<1,$ there exist no non-trivial holomorphic solutions of  $(\ref{2term10})$  satisfying the growth condition
$$\liminf_{r\rightarrow\infty}\frac{\kappa(r)r^2}{\min\{\mathfrak T_{f}(r), \mathfrak T_{g}(r)\}}=0,$$
where $\kappa$ is defined by $(\ref{kappa}).$  In particular,  there are no non-trivial holomorphic solutions  for  $1/m+1/n< 1$ if  $\tilde S=\mathbb C.$
\end{theorem}
\begin{proof} By the definition, we have 
$$\mathfrak T_{f}(r)=T(r,f), \ \ \mathfrak T_{g}(r)=T(r,g).$$
So, it yields from Lemma \ref{lem5} that 
$$\delta^{[1]}(f^m, 0)+\delta^{[1]}(g^n, 0)\leq 1.$$
Since 
$$\limsup_{r\rightarrow\infty}\frac{N^{[1]}(r,1/f^m)}{T(r, f^m)}=\limsup_{r\rightarrow\infty}\frac{N^{[1]}(r,1/f)}{mT(r, f)}\leq\frac{1}{m},$$
then $$\delta^{[1]}(f^m,0)\geq1-\frac{1}{m}.$$
Similarly, 
 $$\delta^{[1]}(g^n,0)\geq1-\frac{1}{n}.$$
Combine the above, it follows that 
$$\frac{1}{m}+\frac{1}{n}\geq1.$$
The proof is completed.
\end{proof}

\begin{theorem}\label{} For $1/m+1/n+1/l< 1/2,$ there exist no non-trivial holomorphic solutions of  $(\ref{3term10})$  satisfying the growth condition
$$\liminf_{r\rightarrow\infty}\frac{\kappa(r)r^2}{\min\{\mathfrak T_{f}(r), \mathfrak T_{g}(r), \mathfrak T_{h}(r)\}}=0,$$
where $\kappa$ is defined by $(\ref{kappa}).$  In particular,  there are no non-trivial holomorphic solutions  for  $1/m+1/n+1/l <1/2$ if  $\tilde S=\mathbb C.$
\end{theorem}
\begin{proof}
It yields from Lemma \ref{lem5} that 
$$\delta^{[2]}(f^m, 0)+\delta^{[2]}(g^n, 0)+\delta^{[2]}(h^l, 0)\leq 2.$$
Since 
$$\limsup_{r\rightarrow\infty}\frac{N^{[2]}(r,1/f^m)}{T(r,f^m)}\leq\limsup_{r\rightarrow\infty}\frac{2N^{[1]}(r,1/f)}{mT(r, f)}\leq\frac{2}{m},$$
then $$\delta^{[2]}(f^m,0)\geq1-\frac{2}{m}.$$
Similarly, 
 $$\delta^{[2]}(g^n,0)\geq1-\frac{2}{n}, \ \  \ \delta^{[2]}(h^l,0)\geq1-\frac{2}{l}.$$
Combine the above, it follows that 
$$\frac{1}{m}+\frac{1}{n}+\frac{1}{l} \geq\frac{1}{2}.$$
The proof is completed.
\end{proof}
Finally, we treat the generalized  Fermat functional equation (\ref{asdf1}), i.e., 
$$
f_1^{n_1}+\cdots+f^{n_k}_k=1, \ \ k\geq2
$$
on $S.$ 
Apply the similar arguments,  we can obtain 
\begin{theorem}\label{} For $1/n_1+\cdots+1/n_k< 1/(k-1),$ there exist no non-trivial holomorphic solutions of  $(\ref{asdf1})$  satisfying the growth condition
$$\liminf_{r\rightarrow\infty}\frac{\kappa(r)r^2}{\min\{\mathfrak T_{f_{1}}(r), \cdots, \mathfrak T_{f_{k}}(r)\}}=0,$$
where $\kappa$ is defined by $(\ref{kappa}).$  In particular,  there are no non-trivial holomorphic solutions  for  $1/n_1+\cdots+1/n_k <1/(k-1)$ if  $\tilde S=\mathbb C.$
\end{theorem}

\section{Fermat functional equations for small functions}

In this final section, we  treat the  equation (\ref{ggg}) for small functions, i.e., 
$$
\alpha_1f_1^{n_1}+\cdots+\alpha_kf^{n_k}_k=1, \ \ k\geq2
$$
over $S,$ where $\alpha_j$ is a small function with respect to  $f_j$ for $1\leq j\leq k.$ Recall that a meromorphic function $\alpha$ on $S$  is called a small function with respect to
 $\psi$ on $S$ if 
$$\limsup_{r\rightarrow\infty}\frac{T(r,\alpha)}{T(r,\psi)}=0.$$

To investigate non-trivial holomorphic solutions of (\ref{ggg}), we modify Lemma \ref{lem5} as follows
\begin{lemma}\label{lem51}  Let $\psi_0, \cdots, \psi_n$ $(n\geq 1)$ be non-constant meromorphic functions on $S$ satisfying $\delta(\psi_j,\infty)=1$ for $0\leq j \leq n$ as well as  
$$\psi_0+\cdots+\psi_n=0.$$
 If 
$$\liminf_{r\rightarrow\infty}\frac{\kappa(r)r^2}{\min\{T(r, \psi_0),\cdots,T(r,\psi_n)\}}=0,$$
where $\kappa$ is defined by $(\ref{kappa}),$ then
$$\sum_{j=0}^n\delta^{[n]}(\psi_j, 0)\leq n.$$
\end{lemma}
\begin{proof}  The argument is similar as in the proof of Lemma \ref{lem5}.

$a)$ $\psi_0,\cdots,\psi_n$ are linearly independent over $\mathbb C$

 Differentiating
$\psi_0+\cdots+\psi_n=0$ to get 
$$\sum_{j=0}^n\frac{\mathfrak X^\mu\psi_j}{\psi_j}\psi_j=0, \ \ \mu=1,\cdots, n.$$
It yields from Crammer's rule  that  
$\psi_j=\Delta_j/\Delta,$
where 
$$\Delta=\frac{W_\mathfrak{X}(\psi_0,\cdots,\psi_n)}{\psi_0\cdots\psi_n}, \ \ \ 
\Delta_j=\frac{\psi_jW_\mathfrak{X}(\psi_0,\cdots,\psi_{j-1},1,\psi_{j+1},\cdots, \psi_n)}{\psi_0\cdots\psi_n}.$$
By  the First Main Theorem and Lemma \ref{ldl} 
 \begin{eqnarray*}\label{}
m(r,\psi_j)&\leq& m(r,\Delta_j)+m(r,1/\Delta)+O(1) \\
&\leq& m(r,\Delta_j)+m(r,\Delta)+N(r,\Delta)+O(1) \\
&\leq& N(r,\Delta)+S(r) \\
&\leq& \sum_{j=0}^nN^{[n]}(r,1/\psi_j)+n\sum_{j=0}^nN(r,\psi_j)+S(r),
\end{eqnarray*}
where 
$$S(r)=O\Big(\sum_{j=0}^n\log T(r,\psi_j)-\kappa(r)r^2\Big).$$
Therefore, 
\begin{eqnarray*}
T(r)&:=&\max\big\{T(r,\psi_0),\cdots, T(r,\psi_n)\big\} \\
&\leq& \sum_{j=0}^nN^{[n]}(r,1/\psi_j)+(n+1)\sum_{j=0}^nN(r,\psi_j)+S(r).
\end{eqnarray*}
For an arbitrary  $\epsilon>0,$
$$N(r,\psi_j)\leq\big(1-\delta(\psi_j, \infty)+\epsilon\big)T(r,\psi_j)$$
and
$$N^{[n]}(r,1/\psi_j)\leq\big(1-\delta^{[n]}(\psi_j, 0)+\epsilon\big)T(r,\psi_j)$$
holds for $r$  large enough.  Thus, it follows  that 
$$T(r)\leq \sum_{j=0}^n\big(1-\delta^{[n]}(\psi_j, 0)+\epsilon\big)T(r)+(n+1)^2\epsilon T(r)+S(r).$$
This implies that 
$\delta^{[n]}(\psi_0, 0)+\cdots+\delta^{[n]}(\psi_n, 0)\leq n,$
provided with the assumed growth condition.

$b)$  $\psi_0,\cdots,\psi_n$ are linearly dependent over $\mathbb C$

This case can be confirmed similarly  to $b)$ in the proof of Lemma \ref{lem5}.
\end{proof}

\begin{theorem}\label{}  There exist no non-trivial holomorphic solutions for $1/n_1+\cdots+1/n_k< 1/(k-1)$ of  $(\ref{ggg})$  satisfying the growth condition
$$\liminf_{r\rightarrow\infty}\frac{\kappa(r)r^2}{\min\{\mathfrak T_{f_{1}}(r), \cdots, \mathfrak T_{f_{k}}(r)\}}=0,$$
where $\kappa$ is defined by $(\ref{kappa}).$  In particular,  there are no non-trivial holomorphic solutions  for  $1/n_1+\cdots+1/n_k <1/(k-1)$ if  $\tilde S=\mathbb C.$
\end{theorem}
\begin{proof}
Since $\alpha_1, \cdots,\alpha_k$ are small functions, then it leads to   $\delta(\alpha_jf_j,\infty)=1$ for $0\leq j \leq k.$
By Lemma \ref{lem51}, it follows that  
$$\delta^{[k-1]}(\alpha_1f^{n_1}_1, 0)+\cdots+\delta^{[k-1]}(\alpha_kf^{n_k}_k, 0)\leq k-1.$$
On the other hand, we have
$$\limsup_{r\rightarrow\infty}\frac{N^{[k-1]}(r,1/\alpha_1f^{n_1}_1)}{T(r,\alpha_1f^{n_1}_1)}\leq\limsup_{r\rightarrow\infty}\frac{(k-1)N^{[1]}(r,1/f_1)}{n_1T(r,f_1)}\leq\frac{k-1}{n_1}.$$
Then
 $$\delta^{[k-1]}(\alpha_1 f^{n_1}_1,0)\geq1-\frac{k-1}{n_1}.$$
Similarly, 
 $$\delta^{[k-1]}(\alpha_j f^{n_j}_j,0)\geq1-\frac{k-1}{n_j}, \ \  j=2,\cdots, k..$$
Combine the above, it yields  that 
$$\frac{1}{n_1}+\cdots+\frac{1}{n_k} \geq\frac{1}{k-1}.$$
This concludes the proof.
\end{proof}

We don't know  yet  the non-existence  of non-trivial meromorphic solutions for $k\geq 4$ of (\ref{ggg})  over a general  Riemann surface $S,$ and we don't even know
that about non-trivial meromorphic solutions of (\ref{asdf}) for $k\geq 4.$ Learning from some known research results, however, we propose the following conjecture

\noindent\textbf{Conjecture 1.}  \emph{There exist no non-trivial meromorphic solutions for $n\geq k^2$ of  $(\ref{asdf})$  satisfying the growth condition
$$\liminf_{r\rightarrow\infty}\frac{\kappa(r)r^2}{\mathfrak T_{f_1,\cdots,f_k}(r)}=0,$$
where $\kappa$ is defined by $(\ref{kappa}).$  In particular,  there are no non-trivial meromorphic solutions  for  $n\geq k^2$ if  $\tilde S=\mathbb C.$}

More general, we conjecture that 

\noindent\textbf{Conjecture 2.}  \emph{There exist no non-trivial meromorphic solutions for $1/n_1+\cdots+1/n_k\leq 1/k$ of  $(\ref{ggg})$  satisfying the growth condition
$$\liminf_{r\rightarrow\infty}\frac{\kappa(r)r^2}{\mathfrak T_{f_1,\cdots,f_k}(r)}=0,$$
where $\kappa$ is defined by $(\ref{kappa}).$  In particular,  there are no non-trivial meromorphic solutions  for  $1/n_1+\cdots+1/n_k\leq 1/k$ if  $\tilde S=\mathbb C.$}

\vskip\baselineskip

\label{lastpage-01}
\end{document}